\documentclass[a4paper]{amsart}

\usepackage{texmacarxiv}
\usepackage[all]{xy}
\xyoption{web}

\title{Index formulae for integral Galois modules}
\author{Alex Bartel}
\address{Department of Mathematics,
Warwick University,
Coventry CV4 7AL,
UK}
\email{a.bartel@warwick.ac.uk}
\thanks{The first author is supported by a research fellowship of the
Royal Commission for the Exhibition of 1851.}

\author{Bart de Smit}
\address{
Mathematisch Instituut,
Universiteit Leiden, Postbus 9512,
2300 RA Leiden,
The Netherlands}
\email{desmit@math.leidenuniv.nl}

\begin{document}
\maketitle

\begin{abstract}
We prove very general index formulae for integral Galois modules,
specifically for units in rings of integers of number fields, for higher $K$--groups
of rings of integers, and for Mordell-Weil groups of elliptic curves over number fields.
These formulae link the respective Galois module structure to other arithmetic
invariants, such as class numbers, or Tamagawa numbers and Tate--Shafarevich groups.
This is a generalisation of known results on units to other Galois modules
and to many more Galois groups, and at the same time a unification of the
approaches hitherto developed in the case of units.
\end{abstract}

\section{Introduction}
Let $F/K$ be a Galois extension of number fields and $\cO_F$ the ring
of integers of $F$. It is a classical problem to investigate relationships
between $\cO_F^\times$ as a Galois module and class numbers of intermediate
extensions.
One shape that such relationships can take is that of unit index formulae,
such as \cite[Proposition 4.1]{Bart} for elementary abelian Galois groups,
or \cite{Alex} for dihedral groups. The introductions to these
papers contain an overview of some of the history of the problem and further references.

In this work, we generalise these index formulae to a
large class of Galois groups and to various different Galois
modules, namely to units of rings of integers of number fields,
to higher $K$--groups thereof, and to Mordell--Weil groups of elliptic curves
over number fields. In fact, we develop an algebraic machine that produces
such formulae in a great variety of contexts. In the introduction,
we will begin by stating some concrete applications, and will progress to ever
more general results.

\begin{theorem}\label{thm:ellcurvesG20}
  Let $G=C_5\rtimes C_4$ be the semidirect product of a cyclic group of order
  4 acting faithfully on a normal cyclic group of order 5.
  Let $F/K$ be a Galois extension of number fields with Galois group $G$, 
  let $N_0$ be the unique intermediate
  extension of degree 4, $N_1,\ldots,N_4$ distinct intermediate extensions
  of degree 5, and let $E/K$ be an elliptic curve. Let the map
  \[
    f_E : \bigoplus_{i=0}^4E(N_i)/E(K)\longrightarrow E(F)/E(K)
  \]
  be induced by inclusion of each summand. Denote the product of Tamagawa
  numbers of $E$ at the finite places of a number field by $C(E/-)$.
  The remaining notation will be recalled
  in \S\ref{sec:ellcurves}. Assume for simplicity that $E$
  has finite Tate--Shafarevich groups over all intermediate extensions of $F/K$
  (we actually prove an unconditional result, see Remark \ref{rem:uncond}). 
  Then
  \begin{eqnarray*}
    \frac{C(E/F)\#\sha(E/F)\left(C(E/K)\#\sha(E/K)\right)^{4}}
    {\prod_{i=0}^4 C(E/N_i)\#\sha(E/N_i)} =
    5^x\cdot
    \frac{\left[E(F):\sum_{i=0}^4 E(N_i)\right]^2}{|\ker f_E|^2},
  \end{eqnarray*}
  where $x={7\rk E/K - \rk E/N_0 - 3\rk E/N_1}$, and where the sum in the index
 is taken inside $E(F)$.
\end{theorem}
\begin{remark}
The index on the right hand side carries information about the $\Z[G]$-module
structure of $E(F)$.
This information goes far beyond the ranks
in the sense that there exist $\Z[G]$-modules whose complexifications are
isomorphic, but which can be distinguished by the index. Thus,
the theorem says that Tamagawa numbers and Tate--Shafarevich groups control
to some extent the fine integral Galois module structure of the Mordell-Weil
group.
\end{remark}

In fact, already the approach in \cite{Alex} easily generalises to the aforementioned
Galois modules, with the only necessary extra ingredient being a compatibility of
standard conjectures on special values of zeta and $L$--functions with Artin formalism.
However, it uses the classification of all integral representations
of dihedral groups, and therefore does not directly extend to other Galois groups.

The techniques in \cite{Bart} on the other hand generalise to a large class of groups.
But they rely on a direct computation with units and are not immediately applicable
to other Galois modules. In particular, Theorem \ref{thm:ellcurvesG20} is not accessible
by either approach.

The main achievement of the present work is a unification and generalisation
of the approaches of \cite{Alex} and \cite{Bart}. Proposition \ref{prop:fixedphi}
is of central importance in this endeavour. A by-product of this unification will be
a more conceptual proof of the already known index formulae and a better
understanding of the algebraic concepts involved.

\begin{assumption}\label{ass:gp}
In the next three theorems, the combination of a group $G$ and a set $\cU$ of subgroups
of $G$ will be one of the following:
\begin{itemize}
\item $G=C_p\rtimes C_n$ a semidirect product of non-trivial cyclic groups with $p$
a prime and with $C_n$ acting faithfully on $C_p$, and with $\cU$ consisting
of $n$ distinct subgroups of $G$ of order $n$ and the normal subgroup of order $p$; or
\item an elementary abelian $p$--group, with $\cU$ consisting of all index $p$
subgroups; or
\item a Heisenberg group of order $p^3$, where $p$ is an odd prime, and with
$\cU$ consisting of the unique normal subgroup $N$ of order $p^2$ and of $p$
non-conjugate cyclic subgroups of order $p$ that are not contained in $N$; or
\item any other group $G$ and set of subgroups $\cU$ such that the map $\phi$
defined by (\ref{eq:phi}) is an injection of $G$-modules with finite cokernel.
\end{itemize}
\end{assumption}

Below, $R_\Q(G)$ denotes the Grothendieck
group of the category of finitely generated $\Q[G]$--modules. For a $G$-module
$M$, $M_\Q$ and $M_{\Z_p}$ denote the $G$-modules $M\otimes_\Z\Q$
and $M\otimes_\Z\Z_p$, respectively. When no confusion can arise, we will
treat $\Q[G]$--modules synonymously with their image in $R_\Q(G)$.
\namedthm{U}\label{thm:u}
Given $G$ and $\cU$ as in Assumption \ref{ass:gp},
there exists an explicitly computable group homomorphism
\[
\alpha_G:R_\Q(G)\longrightarrow \Q^\times
\]
and an explicit $\Q^\times$--valued map $\beta(H)$ on conjugacy classes
of subgroups of $G$, such that
for any Galois extension $F/K$ of number fields with Galois group $G$ and
for any finite $G$--stable set $S$ of places of $F$ containing the Archimedean ones, we have
\begin{eqnarray}
\lefteqn{\frac{h_S(F)h_S(K)^{|\cU|-1}}{\prod_{H\in\cU}h_S(F^H)} =}\\
& & \left(\alpha_G((\cO^\times_{F,S})_\Q)\cdot\prod_{\fp\in S|_K}\beta(D_{\fp})\cdot\frac{\prod_{H\in\cU}|H|}{|G|^{|\cU|-1}}\right)^{1/2}\cdot
\frac{\left[\cO_{F,S}^\times:\prod_{H\in\cU}\cO^\times_{F^H,S}\right]}{|\ker f_{\cO}|}\nonumber,
\end{eqnarray}
where $S|_K$ denotes the places of $K$ lying below those in $S$, $h_S$ denotes
$S$-class numbers, and where the map 
\begin{eqnarray*}
f_{\cO} : \prod_{H\in\cU}\cO^\times_{F^H,S}/\cO^\times_{K,S}
\longrightarrow \cO^\times_{F,S}/\cO^\times_{K,S}
\end{eqnarray*}
is induced by inclusions
$\cO^\times_{F^H,S}/\cO^\times_{K,S}\hookrightarrow\cO^\times_{F,S}/\cO^\times_{K,S}$, $H\in \cU$.
\endnamedthm

\namedthm{K}\label{thm:k}
Given $G$ and $\cU$ as in Assumption \ref{ass:gp},
there exists an explicitly computable group homomorphism
\[
\alpha_G:R_\Q(G)\longrightarrow \Q^\times
\]
(the same as in Theorem \ref{thm:u}) such that for any Galois extension 
of number fields $F/K$ with Galois group $G$, and for $n\geq 2$,
\begin{eqnarray}
\lefteqn{
\frac{|K_{2n-2}(\cO_{F})||K_{2n-2}(\cO_{K})|^{|\cU|-1}}{\prod_{H\in\cU}|K_{2n-2}(\cO_{F^H})|}
=_{2'}}\\
& & \alpha_G(K_{2n-1}(\cO_F)_\Q)^{1/2}\frac{\left[K_{2n-1}(\cO_F):\sum_{H\in\cU}
 K_{2n-1}(\cO_{F^H})\right]}{|\ker f_K|}\nonumber,
\end{eqnarray}
where $=_{2'}$ denotes equality up to an integer power of 2. Here, we have slightly
abused notation by writing
$\left[K_{2n-1}(\cO_F):\sum_{H\in\cU}K_{2n-1}(\cO_{F^H})\right]$ when we mean
\[
\prod_{p\neq 2} \left[K_{2n-1}(\cO_F)_{\Z_p}:\sum_{H\in\cU}
 K_{2n-1}(\cO_{F^H})_{\Z_p}\right],
\]
and $|\ker f_K|$ instead of
$\prod_{p\neq 2}|\ker f_{K,p}|$, where
\begin{eqnarray*}
f_{K,p} : \bigoplus_{H\in \cU}K_{2n-1}(\cO_{F^H})_{\Z_p}/K_{2n-1}(\cO_{K})_{\Z_p}\longrightarrow
K_{2n-1}(\cO_{F})_{\Z_p}/K_{2n-1}(\cO_{K})_{\Z_p}
\end{eqnarray*}
is induced by inclusions (cf. \S\ref{sec:Kgroups})
\[
K_{2n-1}(\cO_{F^H})_{\Z_p}/K_{2n-1}(\cO_{K})_{\Z_p}\hookrightarrow
K_{2n-1}(\cO_{F})_{\Z_p}/K_{2n-1}(\cO_{K})_{\Z_p},\;\;H\in \cU.
\]

\endnamedthm
Most of the literature on the structure of Mordell--Weil groups of elliptic curves centres of course
around questions about the rank. Here, we address the question: assuming that
we know everything about ranks, what can we say about the finer integral structure
of these Galois modules?
The cleanest statements are obtained if one assumes that
the relevant Tate--Shafarevich groups are finite, but we also derive
an unconditional analogue.

\namedthm{E}\label{thm:e}
Given $G$ and $\cU$ as in Assumption \ref{ass:gp},
there exists an explicitly computable group homomorphism
\[
\alpha_G:R_\Q(G)\longrightarrow \Q^\times
\]
(the same as in Theorems \ref{thm:u} and \ref{thm:k})
such that for any Galois extension of number fields $F/K$ with Galois group $G$ and for any
elliptic curve $E/K$ with $\sha(E/F^H)$ finite for all $H\leq G$, we have
\begin{eqnarray}\label{eq:mainE}
\frac{C(E/F)\#\sha(E/F)\left(C(E/K)\#\sha(E/K)\right)^{|\cU|-1}}
{\prod_{H\in \cU}C(E/F^H)\#\sha(E/F^H)} =\\
\alpha(E(F)_\Q)
\frac{\left[E(F):\sum_{H\in \cU}E(F^H)\right]^2}{|\ker f_E|^2}\nonumber,
\end{eqnarray}
where the map
\[
f_E : \bigoplus_{H\in \cU}E(F^H)/E(K)\longrightarrow
E(F)/E(K)
\]
is induced by inclusions $E(F^H)/E(K)\hookrightarrow E(F)/E(K)$, $H\in \cU$.
\endnamedthm
\begin{remark}
Artin's induction theorem implies that given any $\Q[G]$--module $V$,
$\alpha(V)$ is determined by the dimensions
$\dim V^H$ for all $H\leq G$. Therefore, Theorem
\ref{thm:ellcurvesG20} is a special case of Theorem \ref{thm:e}, while
Theorem \ref{thm:u} is a direct generalisation of \cite[Proposition 4.1]{Bart}
and \cite[Theorem 1.1]{Alex}.
\end{remark}
\begin{remark}\label{rem:uncond}
Combining equations (\ref{eq:ellcurves}) and (\ref{eq:final}) yields an
unconditional version of (\ref{eq:mainE}).
Unlike in the case of units, this is, to our knowledge,
the first such index formula for elliptic curves.
\end{remark}
\begin{remark}
It seems rather striking, that the function $\alpha$ is the same in all three theorems.
Thus, it is not only independent of the realisation of $G$ as a Galois group (and in particular of the
base field), but even independent of the Galois module in question.
In \S\ref{sec:examples},
we give an explicit example of how to compute $\alpha$ for a concrete group,
thereby deducing Theorem \ref{thm:ellcurvesG20} from Theorem \ref{thm:e}.
The function $\beta$ in Theorem \ref{thm:u} is trivial on cyclic subgroups.
In particular, the corresponding product
vanishes when $S$ is just the set of Archimedean places.
\end{remark}
\begin{remark}
Brauer relations have been used by many authors to obtain relations between
arithmetic invariants of number fields, of algebraic varieties, etc. See e.g.
\cite{KR-94}, \cite{BB-04}, and the references therein, to name but a few works.
Here, the focus is on using Brauer relations to obtain explicit Galois module
theoretic information.
\end{remark}

The above theorems are special cases of the representation theoretic
machine we develop. In its maximal generality, our result may be stated as
follows (the necessary concepts, particularly that of Brauer relations and of
regulator constants, will be recalled in the next section):
\namedthm{R}\label{thm:R}
Let $G$ be a finite group, and let $(H_i)_i$, $(H_j')_j$ be two finite
sequences of subgroups of $G$ such that there is an injection
\[
\phi:P_1=\bigoplus_i \Z[G/H_i]\longrightarrow \bigoplus_j \Z[G/H_j']=P_2
\]
of $\Z[G]$--permutation modules with finite cokernel (thus,
$\Theta = \sum_i H_i - \sum_j H_j'$ is a Brauer relation in the sense of
Definition \ref{def:BrauerRel}).
For a finitely generated $\Z[G]$--module $M$, write $(P_i,M) = \Hom_G(P_i,M)$,
$i=1,2$, and $(\phi,M)$ for the induced map $(P_2,M)\rightarrow (P_1,M)$.
Finally, denote by $\cC_{\Theta}(M)$ the regulator constant of $M$ with respect
to the Brauer relation $\Theta = \sum_i H_i - \sum_j H_j'$
(see Definition \ref{def:RegConst}). Then,
the quantity
\[
\cC_{\Theta}(M)\cdot\frac{|\coker(\phi,M)|^2}{|\ker(\phi,M)|^2}\frac{|(P_2,M)_{\tors}|^2}{|(P_1,M)_{\tors}|^2}
\]
only depends on the isomorphism class of the $\Q[G]$-module $M\otimes_\Z\Q$.
\endnamedthm

The significance of this result is that regulator constants of Galois modules can be linked
to quotients of their regulators, which in turn are linked to other number theoretic
invariants through special values of $L$--functions. On the other hand,
by making judicious choices of $\phi$, one can turn the cokernel in
Theorem \ref{thm:R} into an index, such as e.g. the one in equations
(U), (K) and (E), or the index of the image of the $G$-module $M$
in $\prod_{H\in \cU}M^H$ under the norm maps, or other natural invariants
of Galois modules.
The function $\alpha$ and the map $f$ then change, depending on the particular
index that one chooses to investigate,
but the left hand side of the equations (U), (K) and (E) does not.

The main object of study will be regulator constants, as defined in
\cite{tamroot}. We begin by recalling in \S \ref{sec:numthry}
how certain quotients of classical regulators of number fields, of Borel
regulators, and of regulators of elliptic curves can be translated into
regulator constants.
In \S\ref{sec:regconsts} and \S\ref{sec:indices}, we will set up a framework
that allows one to translate
regulator constants into indices in a purely representation theoretic
setting. The main new input is Proposition \ref{prop:fixedphi}, which is the
generalisation of \cite[end of \S 4]{Alex} to arbitrary finite groups. With that
in place, we can essentially follow the strategy of \cite[\S 4]{Bart}.
This procedure works under a certain condition on the Galois group, which is
satisfied e.g. by all the groups listed in Assumption \ref{ass:gp}
(see beginning of \S \ref{sec:indices}).
It is an interesting purely group theoretic problem to determine all groups
that satisfy this condition, which we will not address here. Theorems \ref{thm:u},
\ref{thm:k} and \ref{thm:e}
are obtained by simply substituting the regulator constant--index relationship,
as expressed by (\ref{eq:final}),
in any of the number theoretic situations discussed in the next section.

The quotients of number theoretic data
that appear in all these theorems come from Brauer relations
(see \S\ref{sec:numthry}). Brauer proved that any such quotient of numbers of
roots of unity in number fields is a power of two. As a completely independent result, we show
in Proposition \ref{prop:Ktorsion} that the same is true for $(K_{2n-1})_{\tors}$,
which is the analogue of numbers of roots of unity that appears in Lichtenbaum's
conjecture.
\begin{acknowledgements}
A significant part of this research was done while both authors took part in
the 2011 conference on Galois module structures in Luminy. We thank the organisers
of the conference for bringing us together and the CIRM for hosting the conference.
We would also like to thank Haiyan Zhou for pointing out two inaccuracies in
an earlier draft.
\end{acknowledgements}

\section{Artin formalism and regulator constants}\label{sec:numthry}

We begin by recalling the definitions of Brauer relations and of regulator
constants and their relevance for number theory.
\begin{definition}\label{def:BrauerRel}
Let $G$ be a finite group. We say that the formal linear combination
$\sum_H n_H H$ of subgroups of $G$ is a \emph{Brauer relation} in $G$
if the virtual
permutation representation $\oplus_H \C[G/H]^{\oplus n_H}$ is zero.
\end{definition}

\begin{notation}\label{not:tors}
For any abelian group $A$, we write $\overline{A}$ for $A/A_{\tors}$.
For any homomorphism of abelian groups $f:A\rightarrow B$,
write $\overline{f}$ for the induced homomorphism $\overline{A}\rightarrow \overline{B}$
and $f_{\tors}$ for the restriction $A_{\tors}\rightarrow B_{\tors}$.
\end{notation}

\begin{definition}\label{def:RegConst}
Let $G$ be a finite group and $M$ a finitely generated $\Z[G]$--module.
Let $\langle\cdot,\cdot\rangle:M\times M\rightarrow \C$ be a bilinear $G$--invariant
pairing that is non--degenerate on $\overline{M}$. Let $\Theta=\sum_{H\leq G} n_H H$
be a Brauer relation in $G$.
Define the regulator constant of $M$ with respect to $\Theta$ by
\[
\cC_{\Theta}(M) = \prod_{H\leq G}\det\left(\frac{1}{|H|}\langle\cdot,\cdot\rangle|\overline{M^H}\right)^{n_H}\in \C^{\times},
\]
where each determinant is evaluated on any $\Z$-basis of $\overline{M^H}$.
\end{definition}
This definition is independent of the choice of pairing \cite[Theorem 2.17]{tamroot}.
In particular, since one can always choose a $\Q$-valued pairing, $\cC_\Theta(M)$
lies in $\Q^\times$.

Regulator constants naturally arise in number theory. The heart of the paper
will be an investigation of the regulator constants themselves and their relationship
with indices of fixed submodules in a given module. The results
can then be applied to any of the equations (\ref{eq:units}), (\ref{eq:Kgroups}),
or (\ref{eq:ellcurves}) to derive immediate number theoretic consequences.

For the rest of the section, fix
a Galois extension $F/K$ of number fields with Galois group $G$ and
a Brauer relation $\Theta=\sum_H n_H H$ in $G$. We embed $F$ and all other extensions
of $K$ that we will consider inside a fixed algebraic closure of $K$.
\subsection{Units in rings of integers of number fields}
Let $S$ be a finite $G$--stable set of places of $F$ containing all the Archimedean
ones.
Then Artin formalism for Artin $L$--functions, combined with the analytic class number
formula, implies that
\[
\prod_H \left(\frac{h_S(F^H)\Reg_S(F^H)}{w(F^H)}\right)^{n_H} = 1,
\]
where $h_S$ denotes $S$--class numbers, $\Reg_S$ denotes $S$--regulators and
$w$ denotes the numbers of roots of unity.
One can show \cite[Proposition 2.15]{Alex} that
\[
\cC_\Theta(\cO_{S,F}^\times) = \frac{\cC_\Theta(\triv)}{\prod\limits_{\fp\in S|_K}\cC_\Theta(\Z[G/D_\fp])}\cdot
\prod_{H\leq G} \Reg_S(F^{H})^{2n_H},
\]
where $S|_K$ is the set of primes of $K$ lying below those in $S$, and
$D_{\fp}$ denotes the decomposition group of a prime of $F$ above $\fp$.
It follows that
\begin{eqnarray}\label{eq:units}
\prod_H\left(\frac{w(F^H)}{h_S(F^H)}\right)^{2n_H}\cdot
\frac{\cC_\Theta(\triv)}{\prod\limits_{\fp\in S|_K}\cC_\Theta(\Z[G/D_\fp])}
=\cC_\Theta(\cO_{S,F}^\times).
\end{eqnarray}
The factor $\cC_\Theta(\triv)\big/\prod_{\fp\in S|_K}\cC_\Theta(\Z[G/D_\fp])$
can be made very explicit for any concrete $G$. For example, regulator constants
of permutation modules corresponding to cyclic subgroups are always trivial
\cite[Lemma 2.46]{tamroot}, and in particular the denominator vanishes if $S$
is the set of Archimedean places of $F$. Otherwise, it is an easily computable function of
the number of primes in $S$ with given splitting behaviour. It is always a
rational number, which has non--trivial $p$--adic order only for those primes $p$
that divide $|G|$. See \cite[Theorem 1.2]{Alex} for an even stronger restriction.

\subsection{Higher $K$--groups of rings of integers}\label{sec:Kgroups}
Let $\sigma_i:F\hookrightarrow \R$, $1\leq i\leq r_1$ be the real places of $F$
and $\sigma_{j}:F\hookrightarrow \C$, $r_1+1\leq j\leq r_1+r_2$ be representatives
of the complex places. Let $n\geq 2$, and set $d=r_1+r_2$ if $n$ is odd and $d=r_2$
if $n$ is even. Borel has constructed maps
\[
K_{2n-1}(\cO_F)\stackrel{\oplus(\sigma_i)_*}{\longrightarrow}\bigoplus_{i=1}^d K_{2n-1}(\C)\stackrel{\oplus B_n}{\longrightarrow}\bigoplus_{i=1}^d\R
\]
whose composition has kernel precisely equal to the torsion subgroup,
and showed \cite{Bor-74} that the image is a full rank lattice. The $n$--th
Borel regulator $\Reg_n(F)$ is then defined as the covolume of this image.

If $F/K$ and $G$ are as before, and $H$ is a subgroup of $G$, then for any odd prime $p$,
we have
\[
(K_{2n-1}(\cO_F)_{\Z_p})^H \cong K_{2n-1}(\cO_{F^H})_{\Z_p}
\]
(see e.g. \cite[Proposition 2.9]{Kol-02}). We will therefore treat
$K_{2n-1}(\cO_{F^H})_{\Z_p}$ as a subgroup of $K_{2n-1}(\cO_F)_{\Z_p}$.
We have suppressed all the details concerning the normalisation of $B_n$,
for which we refer to \cite{Gil-01}, but note
that these details will be irrelevant for us.
What is important, is that if for a number field $N$ we define a pairing
$\langle\cdot,\cdot\rangle_N$ on $K_{2n-1}(\cO_N)$ by
\[
\langle u,v\rangle_N = \sum_{i=1}^d B_n((\sigma_i)_*u) B_n((\sigma_i)_*v),
\]
then $\langle\cdot,\cdot\rangle_F$ is $G$--invariant, and for $N\leq F$ we have
$\langle\cdot,\cdot\rangle_F|_{K_{2n-1}(\cO_N)} = [F:N]\langle\cdot,\cdot\rangle_N$. Moreover,
by definition, $\Reg_n(N)^2 = \det(\langle u_i,u_j\rangle_N)$ where $u_i,u_j$
range over a basis of $\overline{K_{2n-1}(\cO_N)}$. The Lichtenbaum conjecture on
leading coefficients of Dedekind zeta functions at $1-n$ together with aforementioned
Artin formalism predicts that
\begin{eqnarray}\label{eq:KArtin}
\prod_H\left(\frac{\Reg_n(F^H)|K_{2n-2}(\cO_{F^H})|}{|K_{2n-1}(\cO_{F^H})_{\tors}|}\right)^{n_H}=_{2'}1,
\end{eqnarray}
where $=_{2'}$ denotes equality up to an integer power of 2. In fact,
the Lichtenbaum conjecture is known to be compatible with Artin formalism
\cite{Bur-10}, so equation (\ref{eq:KArtin}) is true unconditionally.
In view of the above discussion, we also get unconditionally
\begin{eqnarray}\label{eq:Kgroups}
\prod_H\left(\frac{|K_{2n-1}(\cO_{F^H})_{\tors}|}{|K_{2n-2}(\cO_{F^H})|}\right)^{2n_H} =_{2'} \cC_\Theta(K_{2n-1}(\cO_F)).
\end{eqnarray}

\subsection{Torsion in higher $K$-groups of rings of integers}
Brauer proved in \cite{Bra-51} that the quotient $\prod_H w(F^H)^{n_H}$
appearing in (\ref{eq:units}) is a power of 2 for any Galois extension
$F/K$ with Galois group $G$ and for any Brauer relation $\sum_H n_HH$.

It is natural to ask whether there is an analogue of this result for higher
$K$-groups. We answer this question affirmatively.
The following result is completely independent of the rest of the paper:
\begin{proposition}\label{prop:Ktorsion}
Let $G$ be a finite group and let $\Theta=\sum_H n_HH$ be a Brauer relation in $G$.
Let $F/K$ be a Galois extension of number fields with Galois group $G$. Then
\[
\prod_H |K_{2n-1}(\cO_{F^H})_{\tors}|^{n_H}
\]
is a power of 2.
\end{proposition}
\begin{proof}
It suffices to show that the rational number in the Proposition
has trivial $p$-part for all odd primes $p$.
Recall (e.g. \cite{Wei-05}) that for a number field $N$ and an odd prime number $p$, the $p$-part of
$K_{2n-1}(\cO_N)_{\tors}$ is isomorphic to $\mu_{p^\infty}(n)^{G_N}$,
the fixed submodule under the absolute Galois group of the $n$--th Tate twist
of the group of $p$--power roots of unity.

Putting $W=\mu_{p^\infty}(n)^{G_F}$, we now make two observations.  First, for
every subgroup $H$ of $G$, the $p$-part of $K_{2n-1}(\cO_{F^H})_{\tors}$ is
isomorphic to $W^H$. Secondly, since $p$ is odd, the automorphism group of $W$
is cyclic. So if $U$ denotes the kernel of the map $G\rightarrow \Aut W$,
then $U$ is normal in $G$, and $G/U$ is cyclic, and in particular has no
non-trivial Brauer relations. Since for $H\leq G$, $\#W^H = \#W^{UH}$, the
result follows immediately from \cite[Theorem 2.36 (q)]{tamroot}.
\end{proof}

The following immediate consequence is noteworthy, since it greatly generalises
several previous works on tame kernels (see e.g. \cite{Wu-11,Zhou-09}):
\begin{corollary}
Let $F/K$ be a finite Galois extension of totally real number fields with
Galois group $G$, let $\Theta=\sum_H n_H H$ be a Brauer relation in $G$. Then
\[
\prod_H |K_{4n-2}(\cO_{F^H})|^{n_H}
\]
is a power of 2 for any $n\geq 1$.
\end{corollary}
\begin{proof}
Since $F$ is totally real, $K_{4n-1}(\cO_N)$ is torsion for any subfield $N$
of $F$. The assertion therefore follows from equation (\ref{eq:KArtin})
and Proposition \ref{prop:Ktorsion}.
\end{proof}

We also get relationships between the finite
even indexed $K$-groups over various intermediate fields without assuming that
$F$ is totally real:
\begin{corollary}
Let $F/K$ be a finite Galois extension of number fields with Galois group $G$,
let $\Theta=\sum_H n_H H$ be a Brauer relation in $G$. Suppose that $N\unlhd G$
is such that $G/N$ is cyclic. Then
\[
\sum_H n_H \ord_p\left(|K_{4n-2}(\cO_{F^H})|\right) = 0
\]
for any prime $p\nmid 2|N|$ and for any integer $n\geq 1$.
\end{corollary}
\begin{proof}
This is an immediate consequence of \cite[Proposition 3.9]{Alex} combined with
equation (\ref{eq:Kgroups}) and Proposition \ref{prop:Ktorsion}.
\end{proof}

\subsection{Mordell--Weil groups of elliptic curves}\label{sec:ellcurves}
Let $E/K$ be an elliptic curve. A consequence
of the Birch and Swinnerton--Dyer conjecture and of Artin formalism for twisted
$L$--functions is
\begin{eqnarray}\label{eq:ellcurvescond}
\prod_{H\leq G}\left(\frac{C(E/F^H)\Reg(E/F^H)\#\sha(E/F^H)}{|E(F^H)_{\tors}|^2}\right)^{n_H} \stackrel{?}{=} 1,
\end{eqnarray}
where $C(E/F^H)$ denotes the product of suitably normalised Tamagawa numbers
over all finite places of $F^H$. See e.g. the introduction to \cite{Bar-09}
for the details on the normalisation of the Tamagawa numbers, which will be
immaterial for us.

Formula (\ref{eq:ellcurvescond}) can in fact be shown to be true only under
the assumption that the relevant Tate--Shafarevich
groups are finite: if we write $\Theta$ as $\sum_i H_i - \sum_j H_j'$,
then the products of Weil restrictions of scalars $\prod_i W_{F^{H_i}/K}(E)$ 
and $\prod_j W_{F^{H_j'}/K}(E)$ are isogenous abelian varieties. The claim 
therefore follows from the compatibility of the Birch and Swinnerton--Dyer
conjecture with Weil restriction of scalars \cite{Mil-72}
and with isogenies \cite[Chapter I, Theorem 7.3]{Mil-06}.

Moreover, if one only assumes that the $p$-primary parts of the Tate-Shafarevich
groups are finite, then the $p$-part of equation (\ref{eq:ellcurvescond}) holds,
in the sense that
\begin{eqnarray}\label{eq:ellcurvescondp}
\sum_{H\leq G}n_H\ord_p\left(\frac{C(E/F^H)\#\sha(E/F^H)[p^\infty]}{|E(F^H)[p^\infty]|^2}\right) =\\
\ord_p\left(\prod_{H\leq G}\Reg(E/F^H)^{-n_H}\right).\nonumber
\end{eqnarray}
Note that substituting the N\'eron--Tate height pairing on $E(F)$ in the definition
of regulator constants yields
\begin{eqnarray}\label{eq:ellregconst}
\cC_{\Theta}(E(F))=\prod_{H\leq G}\Reg(E/F^H)^{n_H}.
\end{eqnarray}
In particular, the product of regulators on the right hand side in equation
(\ref{eq:ellcurvescondp}) is a rational number
(see below), so it does make sense to consider its $p$-adic valuation.

One can also get an entirely unconditional statement by incorporating
the (conjecturally trivial) divisible parts of the Tate--Shafarevich groups
as follows. Let
\[
\phi:\bigoplus_i \Z[G/H_i]\hookrightarrow \bigoplus_j\Z[G/H_j']
\]
be an inclusion of $\Z[G]$--modules with finite cokernel (cf. \S \ref{sec:regconsts}).
Let $A$ and $B$ denote the abelian varieties $\prod_i W_{F^{H_i}/K}(E)$ and
$\prod_j W_{F^{H_j'}/K}(E)$, respectively.
Let $A^t$ and $B^t$ be the dual abelian varieties. Then $\phi$ induces an
isogeny $\phi:A\rightarrow B$ and the dual isogeny
$\phi^t:B^t\rightarrow A^t$. These in turn give maps on the divisible parts of
Tate--Shafarevich groups:
$\phi:\sha(A/K)_\text{div}\rightarrow \sha(B/K)_\text{div}$ and
$\phi^t:\sha(B^t/K)_\text{div}\rightarrow \sha(A^t/K)_\text{div}$. Denote their (necessarily
finite) kernels by $\kappa$ and $\kappa^t$, respectively.
\begin{proposition}\label{prop:bsd}
We have
\[
\prod_{H\leq G}\left(\frac{C(E/F^H)\Reg(E/F^H)\prod_{q\big||G|}\#\sha_0(E/F^H)[q^\infty]}{|E(F^H)_{\tors}|^2}\right)^{n_H}=
\frac{|\kappa^t|}{|\kappa|},
\]
where $\sha_0$ denotes the quotients by the divisible parts.
\end{proposition}
\begin{proof}
See \cite[Theorem 4.3]{squarity} or \cite[\S 4]{Bar-09}
\end{proof}

Combining Proposition \ref{prop:bsd} with equation (\ref{eq:ellregconst}),
we obtain the unconditional statement
\begin{eqnarray}\label{eq:ellcurves}
\prod_{H\leq G}\left(\frac{|E(F^H)_{\tors}|^2|\kappa^t|}{C(E/F^H)\prod_{q\big||G|}\#\sha_0(E/F^H)[q^\infty]|\kappa|}\right)^{n_H}=\cC_\Theta(E(F)).
\end{eqnarray}

\section{Reinterpretation of regulator constants}\label{sec:regconsts}

We retain Notation \ref{not:tors}.
In light of the previous section, our aim is to express regulator constants
of an arbitrary $\Z[G]$--module in terms of the index of a submodule generated
by fixed points under various subgroups of $G$. This will be done in the next
section. Here, we prove the necessary preliminary results.
Throughout this section, $G$ is a finite group,
$\Theta = \sum_i H_i -\sum_j H_j'$ is a Brauer relation in $G$, and $M$ is
a finitely generated $\Z[G]$--module.

We begin by giving an alternative definition of regulator constants.
Write $P_1 = \oplus_i \Z[G/H_i]$, $P_2 = \oplus_j \Z[G/H_j']$.
Since $\Theta$ is a Brauer relation, there exists an injection of $\Z[G]$--modules
\[
\phi:P_1\hookrightarrow P_2
\]
with finite cokernel. Dualising this and fixing isomorphisms between the
permutation modules and their duals, we also get
\[
\phi^{\tr}: P_2 \hookrightarrow P_1.
\]
Note that if $\phi$ is given by the matrix $X$ with respect to some fixed bases
on $P_1$ and $P_2$, then $\phi^{\tr}$ is given by $X^{\tr}$ with respect to the
dual bases. Applying the contravariant functor $\Hom_G(\cdot\;,M)$, which we
abbreviate to $(\cdot\;,M)$, we obtain the maps
\begin{eqnarray*}
(\phi,M):(P_2,M) & \longrightarrow & (P_1,M)\text{ and }\\
(\phi^{\tr},M):(P_1,M) & \longrightarrow & (P_2,M).
\end{eqnarray*}

\begin{lemma}\label{lem:regconsts}
We have
\[
\cC_{\Theta}(M) = \frac{\#\coker (\phi^{\tr},M)\big/\#\ker(\phi^{\tr},M)}{\#\coker (\phi,M)\big/\#\ker(\phi,M)}\cdot
\frac{|(P_1,M)_{\tors}|^2}{|(P_2,M)_{\tors}|^2}.\]
In particular, the right hand side is independent of the $\Z[G]$--module
homomorphism $\phi$.
\end{lemma}
\begin{proof}
Consider the commutative diagram
\[
\xymatrix{
0 \ar[r] & (P_2,M)_{\tors} \ar[d]^{(\phi,M)_{\tors}}\ar[r] & (P_2,M) \ar[d]^{(\phi,M)}\ar[r] & (P_2,M)/\tors \ar[d]^{\overline{(\phi,M)}}\ar[r] & 0\\
0 \ar[r] & (P_1,M)_{\tors} \ar[r] & (P_1,M) \ar[r] & (P_1,M)/\tors \ar[r] & 0,
}
\]
and similarly for $(\phi^{\tr},M)$.
By \cite[Theorem 3.2]{Alex},
$\cC_{\Theta}(M)=\frac{\#\coker\overline{(\phi^{\tr},M)}}{\#\coker\overline{(\phi,M)}}$
(note that in \cite{Alex}, $M$ is assumed to be torsion--free, but the proof
extends verbatim to the general case). Since $\ker\overline{(\phi,M)}$ is trivial,
the snake lemma, applied to the above diagram, implies that
\[\#\coker(\phi,M) = \#\coker(\phi,M)_{\tors}\#\coker\overline{(\phi,M)},
\]
and similarly for $(\phi^{\tr},M)$. Since both torsion subgroups are finite,
\[
\frac{\#\coker(\phi,M)_{\tors}}{\#\ker(\phi,M)_{\tors}}=
\frac{\#\ker(\phi^{\tr},M)_{\tors}}{\#\coker(\phi^{\tr},M)_{\tors}}=
\frac{|(P_1,M)_{\tors}|}{|(P_2,M)_{\tors}|},
\]
whence the result follows.
\end{proof}

\begin{proposition}\label{prop:fixedphi}
For any fixed $\phi$, the value of
\begin{eqnarray}
\frac{\#\coker(\phi,M)\cdot\#\coker(\phi^{\tr},M)}
{\#\ker(\phi,M)\cdot\#\ker(\phi^{\tr},M)}
\end{eqnarray}
only depends on the isomorphism
class of $M\otimes \Q$, and not on the integral structure of $M$.
\end{proposition}
\begin{proof}
As in the previous proof, we have
\[\#\coker (\phi,M) = \#\coker(\phi,M)_{\tors}\cdot
\#\coker \overline{(\phi,M)},\]
and similarly for $(\phi^{\tr},M)$, and
\[
{\#\coker(\phi,M)_{\tors}}{\#\coker(\phi^{\tr},M)_{\tors}}=
{\#\ker(\phi^{\tr},M)}{\#\ker(\phi,M)}.
\]
So it is enough to show that
$\#\coker\overline{(\phi,M)}\cdot\#\coker\overline{(\phi^{\tr},M)}$ only depends on
$M\otimes\Q$. For the rest of the proof,
we drop the overline and assume without loss of generality that $M$ is torsion--free.
We have
\[
\#\coker (\phi,M)\cdot\#\coker (\phi^{\tr},M) = \#\coker((\phi,M)(\phi^{\tr},M))
= \#\coker (\phi\phi^{\tr},M),
\]
so it remains to prove that if $P_1=P_2$, then $\#\coker(\phi,M)$ only depends on
$M\otimes \Q$ for torsion--free $M$. Now, $\Hom_G(P_1,M)$ is a full rank
lattice in the vector space $\Hom_G(P_1,M\otimes \Q)$,
and $\#\coker(\phi,M)$, being the
expansion factor of the lattice under $\phi$, does not depend on the choice
of lattice.
\end{proof}
\begin{corollary}\label{cor:cphi}
For any fixed $\phi$, there exists a function
\[c_{\phi}(M)=c_{\phi}(M\otimes\Q)=\frac{\#\coker(\phi,M)\cdot\#\coker(\phi^{\tr},M)}
{\#\ker(\phi,M)\cdot\#\ker(\phi^{\tr},M)}
\]
that is only a function of the collection of numbers $\rk M^H$ as $H$ ranges
over the subgroups of $G$. It is
uniquely determined by its values on the irreducible rational representations
\footnote{Here and elsewhere, ``irreducible rational representation'' refers to
representations that are irreducible over $\Q$, rather than absolutely irreducible
representations that happen to be defined over $\Q$.}
of $G$.
\end{corollary}
\begin{proof}
The right hand side of the equation is clearly multiplicative in direct sums of
representations. The result follows immediately from Proposition \ref{prop:fixedphi}
and Artin's induction theorem.
\end{proof}

\begin{corollary}
For any fixed $\phi$, we have
\[
\cC_{\Theta}(M)\cdot\frac{|\coker(\phi,M)|^2}{|\ker(\phi,M)|^2} =
c_{\phi}(M)\cdot\frac{|(P_1,M)_{\tors}|^2}{|(P_2,M)_{\tors}|^2}.
\]
\end{corollary}
This proves Theorem \ref{thm:R}.
The name of the game will be to choose suitable injections $\phi$, for which
$\#\coker(\phi,M)/\#\ker(\phi,M)$ can be interpreted in terms of natural
invariants of the Galois module $M$.

\section{Indices of fixed submodules}\label{sec:indices}
From now on, we closely follow \cite[\S 4]{Bart}. Suppose that we have
a subset $\cU$ of subgroups of $G$ such that
\[
\Theta = 1 + (|\cU|-1)\cdot G - \sum_{H\in\cU} H,
\]
is a Brauer relation and such that the map
\begin{eqnarray}\label{eq:phi}
\phi: \Z[G/1]\oplus\Z^{\oplus|\cU|} & \rightarrow & \Z\oplus \bigoplus_{H\in\cU} \Z[G/H]\\
(\sigma,0) & \mapsto & (1,(\sigma H)_{H\in\cU}),\nonumber\\
(0,(n_H)_{H\in\cU}) & \mapsto & (0,(n_HN_H)_{H\in\cU}),\;\;\;N_H=\sum_{g\in G/H}g\nonumber
\end{eqnarray}
is an injection of $G$--modules.
See Example \ref{ex:groups} for some families of groups that have such
Brauer relations.
\begin{remark}
It is important to note that, although we usually treat Brauer relations as
elements of the Burnside ring of $G$, so that subgroups of $G$ are only
regarded as representatives of conjugacy classes, the problem of finding a map $\phi$
as above that is injective may depend on the ``right'' choice of 
conjugacy class representatives.
\end{remark}
The map $\phi$ commutes with taking coinvariants, so we get a commutative
diagram, where the horizontal maps are the augmentation maps:
\[
\xymatrix{
\Z[G]\oplus \Z^{\oplus|\cU|} \ar[d]^\phi\ar@{->>}[r] & \Z^{|\cU|+1}\ar[d]^{\phi_G}\\
\Z\oplus\bigoplus_{H\in \cU}\Z[G/H] \ar@{->>}[r] & \Z^{\oplus|\cU|+1}.
}
\]

Finally, applying the contravariant functor $\Hom_G(\cdot,M)$, we get the commutative
diagram with exact rows
\[
\xymatrix{
0 \ar[r] & (M^G)^{\oplus|\cU|+1} \ar[d]^{(\phi_G,M)}\ar[r] & M^G\oplus\bigoplus_{H\in\cU} M^{H} \ar[d]^{(\phi,M)}\ar[r] & \bigoplus_{H\in\cU} M^{H}/M^G \ar[d]^f\ar[r] & 0\\
0 \ar[r] & (M^G)^{\oplus|\cU|+1} \ar[r] & M\oplus (M^G)^{\oplus|\cU|} \ar[r] & M/M^G \ar[r] & 0,
}
\]
where the map $f$ is induced by inclusions $M^{H}/M^G\hookrightarrow M/M^G$,
so that
\[
\#\coker f=[M:\sum_{H\in\cU} M^{H}].
\]
Thus, the snake lemma, together with the results of the previous section,
yield
\begin{eqnarray}\label{eq:final}
\frac{[M:\sum_{H\in\cU} M^{H}]^2}{|\ker f|^2} & = &
\frac{|\coker(\phi,M)|^2}{|\ker(\phi,M)|^2}\cdot\frac{|\ker(\phi_G,M)|^2}{|\coker(\phi_G,M)|^2}\nonumber\\
& = & \frac{c_{\phi}(M)}{\cC_{\Theta}(M)}\cdot
\frac{|M_{\tors}|^2\cdot|M^G_{\tors}|^{2|\cU|-2}}{\prod_{H\in\cU}|M^{H}_{\tors}|^2}\cdot
\frac{\cC_{0}(M)}{c_{\phi_G}(M)}\cdot\frac{|M^G_{\tors}|^{2|\cU|+2}}{|M^G_{\tors}|^{2|\cU|+2}}\nonumber\\
& = & \frac{c_{\phi}(M)}{c_{\phi_G}(M)\cC_{\Theta}(M)}\cdot
\frac{|M_{\tors}|^2\cdot|M^G_{\tors}|^{2|\cU|-2}}{\prod_{H\in\cU}|M^{H}_{\tors}|^2}.
\end{eqnarray}

\begin{remark}
In an arithmetic context, the torsion quotient will generically vanish. For example, if
$K$ is a number fields, $G$ is a fixed finite group, and $E/K$ is an elliptic curve,
then for a generic Galois extension $F/K$ with Galois group $G$,
$E(F)_{\tors} = E(K)_{\tors}$ (so here, we consider $M=E(F)$). More precisely,
$E(F)[p^\infty]=E(K)[p^\infty]$ for all $F/K$ of bounded degree
whenever $p$ is sufficiently large (the implicit bound only depending on $E/K$ and
on the degree of $F/K$).
If, on the other
hand, we set $M=K_{2n-1}(\cO_F)$, then the torsion quotient is a power of 2.
For $n=1$ this is due to Brauer \cite{Bra-51}, and for $n\geq 2$ this is
Proposition \ref{prop:Ktorsion}. For many concrete groups $G$ and relations
$\Theta$, the torsion quotient can be shown to always vanish.

Note also, that once we substitute
the above formula for the regulator constant into (\ref{eq:units}), (\ref{eq:Kgroups}),
or (\ref{eq:ellcurves}), the torsion quotient cancels, so in any case, it is
not present in the final index formulae.
\end{remark}
Combining equation (\ref{eq:final}) with (\ref{eq:units}), (\ref{eq:Kgroups}),
and (\ref{eq:ellcurves}) proves Theorems \ref{thm:u}, \ref{thm:k}, and \ref{thm:e},
respectively, for all groups that
have a Brauer relation of the form (\ref{eq:phi}). See the next section for
several examples of such groups. The function $\alpha([V])$ of Theorems \ref{thm:u},
\ref{thm:k} and \ref{thm:e}
is given on an irreducible representation $V$ by $c_{\phi_G}(M)/c_\phi(M)$ for
any lattice $M$ in $V$, while $\beta(H)=\cC_{\Theta}(\Z[G/H])^{-1}$.

\section{Examples}\label{sec:examples}
\begin{example}\label{ex:groups}
Examples of groups that admit a Brauer relation of the
form (\ref{eq:phi}) include the following:
\begin{itemize}
\item elementary abelian $p$--groups, with $\cU$ being the set of all
subgroups of index $p$. This recovers the main results of \cite[\S 4]{Bart},
and generalises these results to other Galois modules;
\item semidirect products $C_p\rtimes C_n$ with $p$ an odd prime and with
$C_n$ acting faithfully on $C_p$, and with $\cU$ consisting of $C_p$ and of $n$
distinct subgroups of order $n$. This includes the case of \cite[Theorem 1.1]{Alex}
and provides a vast generalisation of that result.
\item Heisenberg groups of order $p^3$, where $p$ is an odd prime,
and where $\cU$ consists of the unique normal subgroup $N$ of order $p^2$ and of
$p$ non-conjugate cyclic groups of order $p$ that are not contained in $N$.
\end{itemize}
\end{example}

We will now demonstrate that the function $\alpha$
really is explicitly computable in any concrete case.
\begin{example}
Let $G=C_5\rtimes C_4$ with $C_4$ acting faithfully on $C_5$.
As mentioned in the previous example,
\[
\Theta = 1 - 4C_4 - C_5 + 4G
\]
is a Brauer relation, and letting $\cU$ consist of $C_5$ and of 4 distinct
cyclic groups of order 4,
\[
\phi: \Z[G/1]\oplus \Z^5\rightarrow \Z\oplus \bigoplus_{H\in \cU}\Z[G/H]
\]
defined by (\ref{eq:phi}) is an injection of $G$--modules with finite cokernel.
Given a $\Z[G]$--module $M$,
the map $(\phi,M)$ corresponding to $\phi$ is given by
\begin{eqnarray*}
M^G\oplus \bigoplus_{H\in \cU} M^H & \rightarrow & M \oplus
\bigoplus_{H\in \cU} M^G\\
(m,(m_H)_H) & \mapsto & (m + \sum_{H\in\cU} m_H, (\tr_{G/H} m_H)_H),
\end{eqnarray*}
while $(\phi^{\tr},M)$ is easily seen to be
\begin{eqnarray*}
M\oplus\bigoplus_{H\in \cU} M^G & \rightarrow &
M^G\oplus\bigoplus_{H\in \cU} M^{H}\\
(m,(m_H)_{H\in \cU}) &\mapsto &
(\tr_{G/1}m,(\tr_{H/1}m + m_H)_{H\in \cU}).
\end{eqnarray*}

The following is a complete list of irreducible rational representations
of $G$:
\begin{itemize}
\item the two 1--dimensional representations $\triv=\chi_1$, $\chi_2$
that are lifted from $G/D_{10}\cong C_2$,
\item the direct sum $\rho$ of the remaining two 1--dimensional complex
representations lifted from $G/C_5\cong C_4$,
\item the 4--dimensional induction $\tau$ of a non--trivial one--dimensional
character from $C_5$ to $G$.
\end{itemize}
In each of these, we need to choose a $G$--invariant lattice. Clearly, $\chi_1$
and $\chi_2$ contain, up to isomorphism, only one $\Z[G]$--module each, which
we will denote by $M_1$, $M_2$ respectively. Next, let $\Gamma_\rho$ be the
non--trivial torsion--free $\Z[C_2]$--module of $\Z$--rank 1. Define $M_\rho$ to be
the lift from
$G/C_5\cong C_4$ of $\Ind_{C_4/C_2}\Gamma_\rho$. This is a $G$--invariant
full rank sublattice of $\rho$. As for $\tau$, it will be simpler to work
with $\tau^{\oplus 4}$, which can be realised as the induction from $C_5$ of
$\Q[C_5]\ominus \Q[C_5/C_5]$. Let $\Gamma_\tau$ be the
$\Z[C_5]$--module $\Z[C_5/1]\big/\langle\sum_{g\in C_5}g\rangle_\Z$
and define $M_\tau$ to be the induction of $\Gamma_{\tau}$ to $G$, so that
$M_\tau\otimes_\Z\Q\cong \tau^{\oplus 4}$.

Next, we need to compute the orders of cokernels $\#\coker(\phi,M)$, $\#\coker(\phi^{\tr},M)$, $\#\coker(\phi_G,M)$,
and $\#\coker(\phi_G^{\tr},M)$ for $M=M_1,M_2,M_\rho,M_\tau$, although in the case
of $M_1$ we can take a short cut:
\begin{itemize}
\item Let $M=M_1$.
Clearly, the left hand side of equation (\ref{eq:final}) is trivial, so
$\alpha([M\otimes\Q])=1/\cC_\Theta(M)=125$. For all the remaining
lattices, we will have $M^G=0$, so the cokernels will be trivial on corestrictions.
\item Let $M=M_2$. Then, $M^{C_5}=M$ and $M^{C_4}=M^G=0$, so $(\phi,M)$ is surjective,
and
$\#\coker(\phi^{\tr},M) = 5$.
\item The same reasoning applies to $M=M_{\rho}$, but since $M_\rho$ has rank 2, we have
$\#\coker(\phi^{\tr},M) = 25$.
\item Let $M=M_\tau$. We have $M^{C_5} = 0$, while for any subgroup
$H$ of order 4, $M_\tau^H$ has $\Z$--rank 4. An explicit computation, either by
hand or using a computer algebra package, yields $\#\coker(\phi,M)=\#\coker(\phi^{tr},M)=5^6$.
\end{itemize}
To summarise, $\alpha([V])$ is given on the irreducible rational representations
of $G$ by
\begin{itemize}
\item $\alpha([\chi_1]) = 125$;
\item $\alpha([\chi_2]) = 1/5$;
\item $\alpha([\rho]) = 1/25$;
\item $\alpha([\tau]) = \alpha(4[\tau])^{1/4} = (5^{-12})^{1/4} = 5^{-3}$.
\end{itemize}
For an arbitrary rational representation $V$ of $G$, the multiplicities
$\langle V,\cdot\rangle$ of
the irreducible rational representations in the direct sum decomposition of $V$
are determined by $\dim V^H$ as $H$ ranges over subgroups of $G$:
\begin{eqnarray*}
\dim V^G & = & \langle V,\chi_1\rangle;\\
\dim V^{D_{10}} & = & \langle V,\chi_1\rangle + \langle V,\chi_2\rangle;\\
\dim V^{C_5} & = & \langle V,\chi_1\rangle + \langle V,\chi_2\rangle + 2\langle V,\rho\rangle;\\
\dim V^{C_4} & = & \langle V,\chi_1\rangle + \langle V,\tau\rangle.
\end{eqnarray*}
Solving for the multiplicities, we deduce, for an arbitrary rational representation
$V$,
\begin{eqnarray*}
\alpha([V]) & = & 5^{3\rk V^G}\cdot 5^{-(\rk V^{D_{10}} - \rk V^G)}\cdot25^{-(\rk V^{C_5} - \rk V^{D_{10}})/2}\cdot
5^{-3(\rk V^{C_4} - \rk V^G)}\\
& = & 5^{7\rk V^G - \rk V^{C_5} - 3\rk V^{C_4}}.
\end{eqnarray*}
\end{example}
\begin{remark}
Such calculations can be simplified further as follows:
\begin{itemize}
\item By Artin's induction theorem, the function $\alpha$ is also determined
uniquely by its values on $\Q[G/C]$ as $C$ ranges over the cyclic subgroups of $G$,
and these may be easier to calculate in specific cases, especially when it is
difficult to classify all irreducible rational representations.
Indeed, given any irreducible rational representation $V$ of $G$, some integer multiple
$V^{\oplus n}$ can be written as a $\Z$-linear combination of such permutation representations.
But since $\alpha([V])$ is a positive real number, the value of $\alpha(n[V])$ uniquely
determines $\alpha([V])$. We have already implicitly used this to calculate $\alpha([\tau])$.
\item Instead of computing $\coker(\phi,M)$ and $\coker(\phi^{\tr},M)$, one may
compute one of the cokernels together with
$\cC_\Theta(M) = \frac{\#\coker(\phi^{\tr},M)}{\#\coker(\phi,M)}$. This is particularly
helpful when combined with the previous observation, since regulator constants of
permutation representations are very easy to compute: using either \cite[Proposition 2.45]{tamroot} or
\cite[Example 2.19]{tamroot}, one finds that
\[
\cC_{\Theta}(\Z[G/C]) = |G|^{1-|\cU|} \prod_{H\in \cU} \prod_{g\in C\backslash G/H}|H^g\cap C|.
\]
Moreover, for the corresponding cokernels on the corestrictions, the regulator constant
is computed with respect to the trivial Brauer relation and is therefore 1,
so that $\#\coker(\phi_G,M) = \#\coker(\phi^{\tr}_G,M)$ for all integral representations
$M$ of $G$.
\end{itemize}
\end{remark}

\end{document}